\documentclass[11pt]{amsart}
\usepackage{amssymb,latexsym,amsmath,amsthm,enumitem,mathrsfs,geometry}
\usepackage{fancyhdr}
\usepackage{color}
\usepackage{stmaryrd}
\usepackage{hyperref}
\hypersetup{
    colorlinks=true,
    linkcolor=red,
    filecolor=magenta,
    urlcolor=cyan,
    citecolor=green
    }
\usepackage{lineno}
\usepackage{diagbox}
\usepackage{array}
\usepackage{tikz}
\usetikzlibrary{arrows}
\usetikzlibrary{positioning}
\usepackage{float}
\usepackage{bbm}
\usepackage{cleveref}
\usepackage{dsfont}
\usepackage{graphicx}
\usepackage{soul}
\usepackage{verbatim}
\usepackage{comment}
\usepackage{caption} 
\usepackage[symbol*]{footmisc}

\newtheorem{theorem}{Theorem}%[section]
\newtheorem{lemma}{Lemma}%[section]

\newtheorem*{theorem*}{Theorem}
\theoremstyle{remark}

\theoremstyle{plain}
\newtheorem*{MT}{Montgomery Theorem (MT)}%[section]
\definecolor{pink}{rgb}{1,.2,.6}
\definecolor{orange}{rgb}{0.7,0.3,0}
\definecolor{blue}{rgb}{.2,.6,.75}
\definecolor{green}{rgb}{.4,.7,.4}
\definecolor{purple}{RGB}{127,0,255}
\numberwithin{equation}{section}
\allowdisplaybreaks

\title[ Zeta Zeros in a narrow vertical box]{Zeta Zeros in a narrow vertical box }
\author[Goldston]{Daniel A. Goldston}
\address{Department of Mathematics and Statistics, San Jose State University}
\email{daniel.goldston@sjsu.edu}

\author[Suriajaya]{Ade Irma Suriajaya}
\address{Faculty of Mathematics, Kyushu University}
\email{adeirmasuriajaya@math.kyushu-u.ac.jp}

\begin{document}
%\tableofcontents

\begin{abstract}
In 1973 Montgomery proved, assuming the Riemann Hypothesis (RH), that asymptotically at least 2/3 of zeros of the Riemann zeta-function are simple zeros. In a previous note \cite{GS25} we showed how RH can be replaced with a general estimate for a double sum over zeros, and this allows one to then obtain results on zeros that are both simple and on the critical line. Here we give a simple proof based on a direct generalization of Montgomery's proof that on assuming all the zeros are in a narrow vertical box between height $T$ and $2T$ of width $b/\log T$ and centered on the critical line, then, if $b=b(T)\to 0$ as $T\to \infty$, we have asymptotically at least 2/3 of the zeros are simple and on the critical line. 
\end{abstract}
\date{\today}

\maketitle
%%%%%%%%%%%%%%%%%%%%%%%%%%%%%%%%
\section{Introduction} In 1973 Montgomery \cite{Montgomery73} proved, assuming the Riemann Hypothesis (RH), that at least 2/3 of the zeros of the Riemann zeta-function are simple. Prior to this, it was not known either unconditionally or on RH that there were even infinitely many simple zeros. Within a few years of this result, Levinson \cite{Levinson74,Levinson75} proved unconditionally that at least 34.74\% of the zeros are simple and on the critical line. The best current unconditional result \cite{Pra20} in 2020 proved that at least $41.7\%$ of the zeros of $\zeta(s)$ are on the critical line, and at least $40.7\%$ of the zeros are on the critical line and simple. In \cite[Theorem 1]{BGST-CL} we proved with Baluyot and Turnage-Butterbaugh the following conditional result.
\begin{theorem}[Baluyot, Goldston, Suriajaya, Turnage-Butterbaugh]\label{thm1} Assuming all of the zeros $\rho = \beta + i\gamma$ of $\zeta(s)$ with $T<\gamma \le 2T$ lie in the region
\begin{equation} \label{Bb} B_b= \{s=\sigma + i t \ : \ \Big|\sigma - \frac12\Big|< \frac{b}{2\log T},\ T< t\le 2T, \quad b=b(T) \to 0 \ \text{as} \ T\to \infty \}.\end{equation} Then for the zeros of $\zeta(s)$ with $T<\gamma \le 2T$, we have asymptotically that at least 2/3 are simple and on the critical line.
\end{theorem}
To obtain this result, the proof used a specially designed \lq \lq Tsang kernel" \cite{Tsang3}. With more work and computation, in \cite[Theorem 2]{BGST-CL} we proved that in place of $b=b(T) \to 0$ in Theorem \ref{thm1} we can take $b=0.3185$, and further that with $b=0.001$ at least $67.25\%$ of the zeros are simple and on the critical line. 

In this note we will give a simpler proof of Theorem \ref{thm1} avoiding the Tsang kernel and following closely Montgomery's proof using just the Fej{\'e}r kernel and no computation. Everything needed in this proof can be found in \cite{Sel46}, \cite{Montgomery73}, and \cite{GaMu78}. Thus, if RH had been replaced by the assumption \eqref{Bb} in Theorem \ref{thm1}, then, within a few years of when Montgomery obtained his results there never was any fundamental obstacle to using the pair correlation method to detect critical zeros; see \cite{GS25}.

%%%%%%%%%%%%%%%%%%%%%%
\section{Convention for counting zeros with multiplicity}
\label{sec:convention}

Analytic functions have their zeros counted by the argument principle, which counts zeros with their multiplicity. Thus a zero $\rho=\beta+i\gamma$ with multiplicity $m_\rho $ is counted as $m_\rho$ zeros. To handle this situation, we replace the \lq\lq set" of zeta-zeros $Z=\{ \rho: \zeta(\rho)=0\}$ with the multiset $\mathcal{Z} = \{ \rho: \zeta(\rho)=0, \ \rho \ \text{has} \ m_\rho\ \text{copies}\}$. We now make the convention that zeros are taken from the multiset $\mathcal{Z}$. Thus, letting $ N(T)$ denote the number of zeros in the critical strip with $0< \gamma \le T$ counted with multiplicity, we have
\begin{equation} \label{N(T)multiset} N(T) = |\{\rho\in\mathcal{Z}: 0<\gamma\le T\}| = \sum_{\substack{\rho \in \mathcal{Z}\\ 0<\gamma\le T}} 1 .
\end{equation}

%%%%%%%%%%%%%%%%%%%%%%
\section{ Counting zeros in the critical strip}

In 1905 von Mangoldt proved that, for $T\ge3$,
\begin{equation} \label{N(T)sum} N(T)= \sum_{\substack{\rho \in \mathcal{Z} \\ 0<\gamma \le T}}1 = \frac{T}{2\pi}\log \frac{T}{2\pi} - \frac{T}{2\pi} + O(\log{T}) , \end{equation}
and we have the useful special cases
\begin{equation} \label{N(T)easy}N(T)~\sim~ \frac{T}{2\pi}\log T \quad \text{as} \ T\to \infty, \qquad \text{and} \qquad N(T+1)-N(T) = O(\log T).
\end{equation}
Using the second estimate in \eqref{N(T)easy}, or directly as in \cite[Lemma, Ch. 15]{Dav2000}, we have
\begin{equation}\label{w(u)estimate} \sum_{\rho\in\mathcal{Z}} \frac{1}{1+ (t-\gamma)^2 }\ll \log(|t|+2).\end{equation}
As an example we will need later, we obtain for $w(u)=4/(4+u^2)$ from \eqref{F(alpha,T)},
\begin{equation} \label{w-estimate} \sum_{\substack{\rho, \rho' \in \mathcal{Z} \\ 0<\gamma,\gamma'\le T}} w(\gamma -\gamma') \ll \sum_{\substack{\rho \in \mathcal{Z}\\0<\gamma \le T}}\sum_{\rho' \in \mathcal{Z}} \frac{1}{1+(\gamma-\gamma')^2} \ll \sum_{\substack{\rho \in \mathcal{Z}\\0<\gamma \le T}}\log\gamma \le N(T)\log T \ll T\log^2T. 
\end{equation}

%%%%%%%%%%%%%%%%%%%%%%
\section{Montgomery's Simple Zero Method without RH}

In \cite[Theorems 2 and 3]{GS25} we proved by elementary reasoning the following extension of Montgomery's RH simple zero argument. Part {\it iii)} below is due to Soundararajan (in private communication). The result also holds with the same proof if the range $0<\gamma,\gamma'\le T$ is replaced with $T<\gamma,\gamma'\le 2T$.
\begin{theorem} \label{thm2} Suppose there exists a constant $\mathbf{C}\ge 1$ such that, as $T\to \infty$,
\begin{equation} \label{thm2input} \sum_{\substack{\rho, \rho' \in \mathcal{Z}\\ 0<\gamma,\gamma'\le T\\ \gamma=\gamma'}}1 \le \Big(\mathbf{C}+o(1)\Big)\frac{T}{2\pi}\log T. \end{equation}
Then asymptotically,
\begin{enumerate}[label={(\roman*)},itemsep=4pt]
\item \label{thm3-i} the proportion of zeros of $\zeta(s)$ which are simple and on the critical line is $\ge 2-\mathbf{C}$,
\item \label{thm3-ii} the average of the proportions of simple zeros and of critical zeros of $\zeta(s)$ is $\ge (3-\mathbf{C})/2$,
\item \label{thm3-iii} the proportion of zeros of $\zeta(s)$ which are either simple or critical or both is $\ge (4-\mathbf{C})/3$.
\end{enumerate}
\end{theorem}

We will use \ref{thm3-i} of Theorem \ref{thm2} to prove Theorem \ref{thm1}.
As we show in the following section, the bound \eqref{thm2input} in the assumption of Theorem \ref{thm2} can also be obtained unconditionally.
Its application, however, requires other assumptions on the zeros, where in our case, the box assumption of Theorem \ref{thm1} suffices. This will be clear as we prove Theorem \ref{thm1} in Section \ref{sec:proofThm1}.

%%%%%%%%%%%%%%%%%%%%%%%%%%%%%%%%%
\section{Unconditional Bound for Close Pairs of Zeros}

\begin{lemma} \label{lem2} For $0\le h\le T$, we have
\begin{equation} \label{GMbound} \sum_{\substack{\rho,\rho' \in \mathcal{Z} \\ 0<\gamma, \gamma'\le T \\ |\gamma'-\gamma| \le h}} 1 \ll \big(1 + h\log T\big) T\log T. \end{equation}
In particular, if $h=0$, then there exists a constant $\mathbf{C}\ge 1$ for which
\begin{equation} \label{thm2nottrivial} \sum_{\substack{\rho,\rho' \in \mathcal{Z} \\ 0<\gamma, \gamma'\le T \\ \gamma =\gamma' }} 1 \le \Big( \mathbf{C} + o(1) \Big) \frac{T}{2\pi}\log T. \end{equation}
\end{lemma}
This lemma and its proof may be found in \cite[Lemma 9]{GM87}. This was proved earlier in {\cite[Sec. 1]{GaMu78}}\footnote[2]{Gallagher and Mueller attribute this to Montgomery's unpublished manuscript \lq\lq Gaps Between Primes" \cite{Montgomery75}.} using a result of Fujii \cite{Fu74} following work of Selberg \cite{Sel46} who discovered the method for obtaining these results unconditionally. We remark here that in \cite[Sec. 1]{GaMu78}, Gallagher and Mueller assumed RH, thus their version of \eqref{GMbound} is obtained by adding their estimates for $N^\ast(T)$ and $N(T,U)$ on the fifth and sixth lines from the bottom in p. 209 of \cite[Sec. 1]{GaMu78}.
Note however that their proof there does not require RH.
Notice that \eqref{thm2nottrivial} shows that unconditionally one can obtain a constant $\mathbf{C}$ for which \eqref{thm2input} holds, although the value by this unconditional method is much larger than 2. Another application of this method for conditionally obtaining critical zeros can be found in \cite{GLSS1}.

%%%%%%%%%%%%%%%%%%%%%%%%%%%%%%%%%
\section{Montgomery's Theorem with and without RH}

In \cite{Montgomery73}, Montgomery introduced the subject of pair correlation of zeros of the Riemann zeta-function through analyses of the function
\begin{equation}\label{F(alpha,T)}
F(\alpha,T) := \sum_{\substack{\rho,\rho' \in \mathcal{Z}\\ 0<\gamma,\gamma'\le T}} T^{i\alpha(\gamma-\gamma')} w(\gamma-\gamma'), \qquad \text{where} \qquad \quad w(u) := \frac{4}{4+u^2},
\end{equation}
with $\alpha$ real and $T\ge 3$.
Montgomery also assumed RH holds for $F(\alpha,T)$, but we will not do that here. Thus the sum in $F(\alpha, T)$ runs over zeros $\rho=\beta+i\gamma$ and $\rho'=\beta'+i\gamma'$, but the weight does not depend on $\beta$ or $\beta'$ at all, and instead depends on $\gamma-\gamma'$. We next define 
\begin{equation}\label{calF(x,T)}
\mathcal{F}(\alpha,T) :=\sum_{\substack{\rho, \rho' \in \mathcal{Z} \\ 0<\gamma,\gamma' \le T}} T^{\alpha(\rho-\rho')}W(\rho-\rho'), \qquad \text{where} \qquad W(u) := \frac{4}{4 - u^2}. 
\end{equation}
Note here that if we assume RH so that $\beta=\beta'=1/2$, we have $\mathcal{F}(\alpha,T)=F(\alpha,T)$. We now state both the unconditional and RH versions of Montgomery's theorem.
\begin{MT}\label{MT}
For $\alpha$ real and $T\ge 3$, we have $\mathcal{F}(\alpha,T)$ is real, $\mathcal{F}(\alpha,T)\ge 0$, $\mathcal{F}(\alpha,T)= \mathcal{F}(-\alpha,T)$, and, uniformly for $0\le \alpha \le 1$ as $T\to \infty$,
\begin{equation}\label{Mon-noRH}
\mathcal{F}(\alpha,T) = \left((1+o(1))T^{-2\alpha}\log T +\alpha +o(1)\right)\frac{T}{2\pi}\log T.
\end{equation}
We also have $F(\alpha,T)$ is real, $F(\alpha,T)\ge 0$, $F(\alpha,T) = F(-\alpha,T)$, and
assuming RH,
\begin{equation}\label{Mon-RH}
F(\alpha,T) = \left((1+o(1))T^{-2\alpha}\log T +\alpha +o(1)\right)\frac{T}{2\pi}\log T,
\end{equation}
uniformly for $0\le \alpha \le 1$ as $T\to \infty$. The theorem also holds if we replace $0<\gamma , \gamma' \le T$
with $T<\gamma,\gamma'\le 2T$ in $\mathcal{F}(x,T)$ and $F(x,T)$.
\end{MT}
The proof of this theorem is in \cite{BGST-PC}, and the proof that we can replace $0<\gamma , \gamma' \le T$
with $T<\gamma,\gamma'\le 2T$ in $\mathcal{F}(x,T)$ and $F(x,T)$ is in \cite{BGST-CL}. The proof of the unconditional version for $\mathcal{F}(\alpha, T)$ in \eqref{Mon-noRH} is the same as Montgomery's original 1973 proof with minor modifications to account for taking the zeros to be $\rho=\beta+i\gamma$. The RH version for $F(\alpha,T)$ is thus just a special case of the unconditional result.
\bigskip

%%%%%%%%%%%%%%%%%%%%%%%%%%%%%%%%%
\section{Unconditional Fej{\'e}r Kernel Sum over Zero Differences without RH}

Montgomery used the Fej{\'e}r kernel to obtain his simple zero result by proving on RH the second part of the following lemma. More recently, Aryan \cite{Ary22} obtained a form of the unconditional part of this lemma.

\begin{lemma}[Aryan]\label{lem1} We have
\begin{equation}\label{fejerlem1}
\sum_{\substack{\rho,\rho' \in \mathcal{Z} \\ 0<\gamma,\gamma'\le T}} \left(\frac{\sin (\frac12 i (\rho-\rho')\log T)}{\frac12 i(\rho-\rho')\log T}\right)^2W(\rho-\rho') 
= \left(\frac43+o(1)\right) \frac{T}{2\pi}\log T .\end{equation} 
If we assume RH, then $\rho -\rho'=i(\gamma -\gamma')$ and we obtain
\begin{equation}\label{fejerlem1RH}\sum_{\substack{\rho,\rho' \in \mathcal{Z} \\ 0<\gamma,\gamma'\le T}} \left(\frac{\sin (\frac12 (\gamma-\gamma')\log T)}{\frac12 (\gamma-\gamma')\log T}\right)^2 w(\gamma -\gamma')
= \left(\frac43+o(1)\right) \frac{T}{2\pi}\log T .\end{equation}
In \eqref{fejerlem1RH} we can remove the weight $w(\gamma-\gamma')$ if we wish. In both results, we can also replace $0<\gamma , \gamma' \le T$
with $T<\gamma,\gamma'\le 2T$.
\end{lemma}
\begin{proof}[Proof of Lemma \ref{lem1}] For a complex number $z$, we have
\[\begin{split}\int_{-1}^1 e^{z\alpha}(1-|\alpha|)\, d\alpha &= \int_0^1 (e^{z\alpha}+e^{-z\alpha})(1-\alpha)\, d\alpha = 2 \int_0^1 \cosh(z\alpha)(1-\alpha)\, d\alpha \\&
= 2\left( \frac{\cosh z -1}{z^2}\right) = \left(\frac{\sinh \frac{1}{2}z}{\frac{1}{2}z}\right)^2 = \left(\frac{\sin \frac{1}{2}iz}{\frac{1}{2}iz}\right)^2. \end{split}\]
Thus, using \eqref{calF(x,T)} we have with $z=(\rho-\rho')\log T$ that
\begin{align*}
    \int_{-1}^1 \mathcal{F}(\alpha,T)(1-|\alpha|)\, d\alpha
    &=
    \sum_{\substack{\rho,\rho' \in \mathcal{Z} \\ 0<\gamma,\gamma'\le T}}
    \left(
    \int_{-1}^1 e^{\alpha(\rho-\rho')\log T}(1-|\alpha|)\, d\alpha
    \right)
    W(\rho-\rho')
    \\
    &=
    \sum_{\substack{\rho,\rho' \in \mathcal{Z} \\ 0<\gamma,\gamma'\le T}}
    \left(\frac{\sin \frac{1}{2}i(\rho-\rho')\log T}{\frac{1}{2}i(\rho-\rho')\log T}\right)^2
    W(\rho-\rho').
\end{align*}
Since $\mathcal{F}(\alpha,T)$ is even in $\alpha$, by \eqref{Mon-noRH} we have as $T\to \infty$,
\[\begin{split}
\int_{-1}^1 \mathcal{F}(\alpha,T)(1-|\alpha|)\, d\alpha
&=
2\int_0^1 \mathcal{F}(\alpha,T)(1-\alpha)\, d\alpha \\
&= 2\left(\int_0^1 \Big( \big(1+o(1)\big)T^{-2\alpha}\log T + \alpha+o(1) \Big) (1-\alpha)\,d\alpha\right) \frac{T}{2\pi}\log T \\
&= \left(\left. T^{-2\alpha}\left( \alpha - 1 + \frac1{2 \log T}\right)\right|_0^1 +\frac13\right) \big(1+o(1)\big)\frac{T}{2\pi}\log T \\
&= \left(\frac43+o(1)\right)\frac{T}{2\pi}\log T,
\end{split}\]
and combining this with the previous equation proves \eqref{fejerlem1}. 

To remove the weight $w(\gamma -\gamma')$, note that $w(u) =1 - \frac14u^2w(u)$, and therefore 
\[ \sum_{\substack{\rho,\rho' \in \mathcal{Z} \\ 0<\gamma,\gamma'\le T}} \left(\frac{\sin (\frac12 (\gamma-\gamma')\log T)}{\frac12 (\gamma-\gamma')\log T}\right)^2 w(\gamma-\gamma') = \sum_{\substack{\rho,\rho' \in \mathcal{Z} \\ 0<\gamma,\gamma'\le T}} \left(\frac{\sin (\frac12 (\gamma-\gamma')\log T)}{\frac12 (\gamma-\gamma')\log T}\right)^2 + O\left( \frac{F(0,T)}{\log^2 T}\right). \]
By \eqref{Mon-RH}, $F(0,T) \ll T\log^2T$, and thus the error term above is $O(T) =o(T\log T)$. (While \eqref{Mon-RH} assumes RH, by \eqref{w-estimate} this estimate holds unconditionally.)

Using $T<\gamma,\gamma'\le 2T$ in \hyperref[MT]{\bf MT}, we see Lemma \ref{lem1} also holds for this range as well.
\end{proof}

%%%%%%%%%%%%%%%%%%%%%%%%%%%%%%%%%
\section{Proof of Theorem \ref{thm1}}
\label{sec:proofThm1}

We first obtain the following bound on the Fej{\'e}r kernel due to Selberg \cite{Sel46}. 
\begin{lemma}[Selberg] \label{lem3} With $z=x+iy$, where $x$ and $y$ are real and $y\ge 0$, 
\begin{equation} \label{SelFbound} 
\left|\left(\frac{\sin(x+iy)}{x+iy}\right)^2 - \left(\frac{\sin x}{x}\right)^2 \right|\ll \frac{ye^{2y}}{1 + x^2 +y^2}.
\end{equation}
\end{lemma}
\begin{proof} This is Selberg's proof \cite[pp. 139--140]{Sel46}. We are bounding the change in the Fej{\'e}r kernel as we move up vertically from the real axis at $z=x$
to $z=x+iy$. Thus 
\[ \begin{split} \left| \left(\frac{\sin(x+iy)}{x+iy}\right)^2 - \left(\frac{\sin x}{x}\right)^2 \right| &= \left| \int_0^y \frac{d}{dt}\left( \left(\frac{\sin(x+it)}{x+it}\right)^2 \right)\,dt \right| \\&
= 2 \left|\int_0^y \frac{\sin(x+it)}{x+it}\left(\frac{(x+it)\cos(x+it)-\sin(x+it)}{(x+it)^2} \right)\,dt \right| \\&
\le 2 \int_0^y \left|\frac{\sin(x+it)}{x+it}\right|\left|\frac{(x+it)\cos(x+it)-\sin(x+it)}{(x+it)^2} \right|\,dt.
\end{split}\]
For $t\ge 0$, we have that 
\[ \left|\frac{\sin(x+it)}{x+it}\right| \ \text{and} \ \left|\frac{(x+it)\cos(x+it)-\sin(x+it)}{(x+it)^2} \right| \ll e^t\min\left\{1, \frac{e^t}{|x+it|}\right\} \ll \frac{e^t}{\sqrt{1+x^2 +t^2 }}, \]
and therefore the bound above is
\[ \left| \left(\frac{\sin(x+iy)}{x+iy}\right)^2 - \left(\frac{\sin x}{x}\right)^2 \right| \ll \int_0^y \frac{e^{2t}}{1+x^2+t^2}\, dt \ll \frac{ye^{2y}}{1 + x^2 +y^2}. \]
\end{proof}
\begin{proof}[Proof of Theorem \ref{thm1}]
Applying Lemma \ref{lem1} with the range $T<\gamma,\gamma'\le 2T$, we have by the assumption in Theorem \ref{thm1} that
\begin{equation}\label{K_b} K_b(T) := \sum_{\substack{\rho,\rho' \in \mathcal{Z} \\ \rho \in B_b, \,  \rho' \in B_b \\ T<\gamma,\gamma'\le 2T}} \left(\frac{\sin (\frac12 i (\rho-\rho')\log T)}{\frac12 i(\rho-\rho')\log T}\right)^2W(\rho-\rho') 
= \left(\frac43+o(1)\right) \frac{T}{2\pi}\log T.
\end{equation} 
Since $W(z)=1 +\frac{z^2}{4-z^2} $, we have
\[ K_b(T) = \sum_{\substack{\rho,\rho' \in \mathcal{Z} \\ \rho \in B_b, \,  \rho' \in B_b \\ T<\gamma,\gamma'\le 2T}} \left(\frac{\sin (\frac12 i (\rho-\rho')\log T)}{\frac12 i(\rho-\rho')\log T}\right)^2 + \sum_{\substack{\rho,\rho' \in \mathcal{Z} \\ \rho \in B_b, \,  \rho' \in B_b \\ T<\gamma,\gamma'\le 2T}} \left(\frac{\sin (\frac12 i(\rho -\rho') \log T)}{\frac12 i\log T}\right)^2\frac{1}{4-(\rho -\rho')^2}. \] 
We note that with $\rho,\rho'\in B_b$, we have $|\beta -\beta'|< b/\log T$, where $b\to 0$ as $T\to \infty $. Therefore, the second sum above is, using \eqref{w-estimate}, 
\[ \ll \frac{e^{b/2}}{\log^2 T} \sum_{\substack{\rho,\rho' \in \mathcal{Z} \\ 0<\gamma,\gamma'\le 2T}} w(\gamma-\gamma') \ll T =o(T\log T). \]

Next, by Lemma \ref{lem3} with $x= -\frac12 (\gamma -\gamma')\log T$ and $y = \frac{1}2(\beta-\beta')\log T $, and assuming $\beta \ge \beta'$, we have 
\[ \begin{split} \left(\frac{\sin (\frac12 i (\rho-\rho')\log T)}{\frac12 i(\rho-\rho')\log T}\right)^2 &= \left(\frac{\sin\Big( \big(-\frac12 (\gamma -\gamma')+ \frac{i}2(\beta-\beta')\big)\log T\Big)}{\big(-\frac12 (\gamma -\gamma')+ \frac{i}2(\beta-\beta')\big)\log T}\right)^2 \\&
= \left(\frac{\sin (\frac12 (\gamma-\gamma')\log T)}{\frac12 (\gamma-\gamma')\log T}\right)^2 + O\left(\frac{ (|\beta-\beta'|\log T) T^{|\beta -\beta'|}}{1 + \big( (\beta-\beta')^2 + (\gamma -\gamma')^2\big)\log^2T}\right).\end{split} \]
Hence 
\[ K_b(T) = \sum_{\substack{\rho,\rho' \in \mathcal{Z} \\ \rho \in B_b, \,  \rho' \in B_b \\ T<\gamma,\gamma'\le 2T}} \left(\frac{\sin (\frac12 (\gamma-\gamma')\log T)}{\frac12 (\gamma-\gamma')\log T}\right)^2 + O\Bigg( \sum_{\substack{\rho,\rho' \in \mathcal{Z} \\ \rho \in B_b, \,  \rho' \in B_b \\ T<\gamma,\gamma'\le 2T}}\frac{T^{|\beta -\beta'|} |\beta-\beta'|\log T}{1 + \big((\beta-\beta')^2 + (\gamma -\gamma')^2\big)\log^2T} \Bigg). \]
The error term here is by Lemma \ref{lem2},
\[\begin{split} &\ll be^b \sum_{\substack{\rho,\rho' \in \mathcal{Z} \\ 0<\gamma,\gamma'\le 2T\\ |\gamma-\gamma'|\leq \frac{1}{\log T}}}\frac{ 1}{1 + \big((\gamma -\gamma')\log T\big)^2} + be^b\sum_{k=1}^\infty \sum_{\substack{\rho,\rho' \in \mathcal{Z} \\ 0<\gamma,\gamma'\le 2T \\ \frac{2^{k-1}}{\log T} < |\gamma-\gamma'|\le \frac{2^k}{\log T}}}\frac{ 1}{1 + \big((\gamma -\gamma')\log T\big)^2} \\&
\ll be^b\sum_{k=1}^\infty \frac{1}{2^{2k}}\sum_{\substack{\rho,\rho' \in \mathcal{Z}\\0<\gamma,\gamma'\le 2T\\ |\gamma-\gamma'|\le \frac{2^k}{\log T}}}1 \ll \ be^b \sum_{k=1}^\infty \frac{2^k T\log T}{2^{2k}} \ll \ be^b T \log T = o(T\log T) ,
\end{split}\]
since $b\to 0$ as $T\to \infty$.
By \eqref{K_b} we conclude that under the assumptions of Theorem \ref{thm1} we have 
\[\begin{split} \sum_{\substack{\rho, \rho' \in \mathcal{Z}\\ T<\gamma,\gamma'\le 2T\\ \gamma=\gamma'}}1 = \sum_{\substack{ \rho, \rho' \in \mathcal{Z} \\ \rho\in B_b,\, \rho'\in B_b \\ T<\gamma,\gamma'\le 2T\\ \gamma=\gamma'}}1 &\le \sum_{\substack{\rho, \rho' \in \mathcal{Z}\\ \rho\in B_b,\, \rho'\in B_b\\ T<\gamma,\gamma'\le 2T}} \left(\frac{\sin (\frac12 (\gamma-\gamma')\log T)}{\frac12 (\gamma-\gamma')\log T}\right)^2 \\ & =K_b(T) +o(T\log T) =\left(\frac43+o(1)\right) \frac{T}{2\pi}\log T.\end{split} \]
Thus $\mathbf{C}=4/3$ holds in Theorem \ref{thm2} and \ref{thm3-i} of Theorem \ref{thm2} immediately implies Theorem \ref{thm1}.
\end{proof}

%%%%%%%%%%%%%%%%%%%%

\section*{Acknowledgment of Funding}

The authors thank the American Institute of Mathematics for its hospitality and providing a pleasant environment where our work on this project began. %During the time this exposition is written and completed,
The second author is currently supported by the Inamori Research Grant 2024, JSPS KAKENHI Grant Number 22K13895, and Kyushu University International Research Leader Training Program (EBXU0101).

\bibliographystyle{alpha}
\bibliography{AHReferences}

\end{document}